\documentclass{amsart}
\usepackage{graphicx}
\usepackage{amsmath}
\usepackage{amssymb}%
\usepackage{amsfonts}%
\usepackage{graphicx}
\usepackage[all]{xy}
\usepackage[mathscr]{euscript}
\usepackage{hyperref}
\newtheorem{theorem}{Theorem}[section]
\newtheorem{lemma}[theorem]{Lemma}

\theoremstyle{definition}
\newtheorem{definition}[theorem]{Definition}

\newtheorem{proposition}[theorem]{Proposition}
\newtheorem{corollary}[theorem]{Corollary}

\usepackage[all]{xy}
\usepackage{graphicx}
\theoremstyle{remark}
\newtheorem{remark}[theorem]{Remark}

\numberwithin{equation}{section}
\usepackage{graphicx}

\begin{document}

\title{\rm On the topology of the spaces of curvature constrained plane curves}

\author{Jos\'{e} Ayala}
\address{FIA, Universidad Arturo Prat, Iquique, Chile}
\email{jayalhoff@gmail.com}
\date{}
\dedicatory{}
\maketitle

\begin{abstract}It is well known that plane curves with the same endpoints are homotopic. An analogous claim for plane
curves with the same endpoints and bounded curvature still remains open. In this work we find
necessary and sufficient conditions for two plane curves with bounded curvature to be deformed, one to
another, by a continuous one-parameter family of curves also having bounded curvature. We conclude that
the space of these curves has either one or two connected components, depending on the distance between
the endpoints. The classification theorem here presented answers a question raised in 1961 by L. E. Dubins.
\end{abstract}

\section{Introduction}

It is well known that any two plane curves with the same endpoints are homotopic. Surprisingly, an analogous claim for plane curves with the same endpoints and a bound on the curvature still remains open. Since the plane is simply connected, all immersed plane curves connecting two different points in the plane are regularly homotopic. On the other hand, H. Whitney in \cite{whitney} classified the regular homotopy classes of closed curves in the plane. A natural step forward is to study homotopy classes in spaces of curvature constrained plane curves. A $\kappa$-constrained curve is required to be smooth with the curvature bounded by a positive constant $\kappa$. In this work we obtain necessary and sufficient conditions for any two $\kappa$-constrained plane curves to be deformed one into another by a continuous one-parameter family of $\kappa$-constrained plane curves. We pay special attention in the interaction between the bound on the curvature and the distance between the endpoints in a $\kappa$-constrained curve. 

Our main result, Theorem \ref{mainthmcbc}, gives the number of homotopy classes in spaces of $\kappa$-constrained plane curves for any choice of initial and final points in $\mathbb R^2$. Let $\kappa=\frac{1}{r}$ be the bound on the curvature (with $r>0$ being the minimum radius of curvature), and, let $d$ be the euclidean distance between the initial and final points in a curve: if $d=0$ then any two closed $\kappa$-constrained plane curves are $\kappa$-constrained homotopic, i.e. are homotopic and satisfy the same curvature bound throughout the deformation; if $0<d<2r$, we prove the existence of two homotopy classes of $\kappa$-constrained plane curves - one homotopy class includes the straight line between the initial and final points, and the other one includes the two outer arcs of circle in the leftmost illustration in Figure \ref{figless2r}; finally if $d>2r$ then any two $\kappa$-constrained plane curves are $\kappa$-constrained homotopic one to another. In addition, if $0<d<2r$ we prove in Theorem \ref{Iconnect} the existence of a planar region where only embedded $\kappa$-constrained plane curves can be defined (see $\mathcal I$ in Figure \ref{figless2r}). Also, there are regions where $\kappa$-constrained plane curves cannot be defined (see $\mathcal E$ in Figure \ref{figless2r} and Figure \ref{figexampcurves}).

Curves with a bound on the curvature and fixed initial and final points and tangent vectors have been extensively studied in science and engineering as well as in mathematics. Most of the papers in this area have been focused on issues related with reachability and optimality cf. \cite{hee, paperb, papera, cockayne, dubins 1, johnson, melzak, reeds, soures}. This paper presents the first results on the spaces of curves with a bound on the curvature where the initial and final vectors are allowed to vary. The classification theorem here presented answers a question raised in 1961 by L. E. Dubins in \cite{dubins 2}.

 \section{Preliminaries}

Throughout this work we consider a parametrised plane curve to be the continuous image in $\mathbb R^2$ of a closed interval.
\begin{definition} \label{cbc} An arc-length parameterised plane curve $\sigma: [0,s]\rightarrow {\mathbb R}^2$ is called a
{\it $\kappa$-constrained plane curve} if:
\begin{itemize}
\item $\sigma$ is $C^1$ and piecewise $C^2$;
\item $||\sigma''(t)||\leq \kappa$, for all $t\in [0,s]$ when defined, $\kappa>0$ a constant.
\end{itemize}
\end{definition}

The first condition in Definition \ref{cbc} means that a $\kappa$-constrained curve has continuous first derivative and piecewise continuous second derivative. The second condition means that a $\kappa$-constrained plane curve has absolute curvature bounded above (when defined) by a positive constant $\kappa=\frac{1}{r}$ where $r>0$ is the minimum radius of curvature. Denote the interval $[0,s]$ by $I$. The length of $\sigma$ restricted to $[a,b]\subset I$ is denoted by ${\mathcal L}(\sigma,a,b)$ and $\mathcal L(\gamma)=s$. Since the curves here studied lie in $\mathbb R^2$ sometimes we refer to a $\kappa$-constrained plane curve just by {\it $\kappa$-constrained curve}. 

The interior, boundary and closure of a subset $X$ in a topological space are denoted by $int(X)$, $\partial(X)$ and $cl(X)$ respectively. The ambient space of the curves here studied is $\mathbb R^2$ with the topology induced by the euclidean metric.

\begin{definition} Given $x,y\in {\mathbb R}^2$. The space of $\kappa$-constrained plane curves from $x$ to $y$ is denoted $\Sigma(x,y)$.
\end{definition}

Throughout this note we consider the space $\Sigma(x,y)$ together with the $C^1$ metric. Suppose a $\kappa$-constrained curve is continuously deformed under a parameter $p$. For each $p$ we reparametrise the corresponding curve by its arc-length. Thus $\sigma: [0,s_p]\rightarrow {\mathbb R}^2$ describes a deformed curve at parameter $p$, and $s_p$ corresponds to its arc-length.

\begin{definition} \label{homcbc} Given $\sigma, \gamma \in \Sigma(x,y)$. A {\it $\kappa$-constrained homotopy} between $\sigma: [0,s_0] \rightarrow \mathbb R^2$ and $ \gamma : [0,s_1] \rightarrow \mathbb R^2$ corresponds to a continuous one-parameter family of immersed plane curves $ { H}: [0,1] \rightarrow \Sigma(x,y)$ such that:
\begin{enumerate}
\item $ { H}(0)=\sigma(t)$ for $t\in [0,s_0]$ and ${H}(1)= \gamma(t)$ for $t\in [0,s_1]$.
\item ${ H}(p): [0,s_p] \rightarrow \mathbb R^2$ is an element of $\Sigma(x,y)$ for all $p\in [0,1]$.
\end{enumerate}
We say that the curves $\sigma$ and $\gamma$ are $\kappa$-{\it constrained homotopic}.
\end{definition}

\begin{remark}{\it Homotopy classes in $\Sigma({x,y})$}. Given $x,y\in {\mathbb R^2}$ then:
\end{remark}
\begin{itemize}
\item two curves are $\kappa$-constrained homotopic if there exists a {\it $\kappa$-constrained homotopy} from one curve to another. The previously described relation defined by $\sim$ is an equivalence relation;
\item a {\it homotopy class} in $\Sigma({x,y})$ corresponds to an equivalence class in $\Sigma(x,y)/\sim$;
\item a {\it homotopy class} is a maximal path connected set in $\Sigma(x,y)/\sim$;
\item we denote by $|\Sigma(x,y)|$ the number of {\it homotopy classes} in $\Sigma(x,y)$.
\end{itemize}

\begin{definition}\label{fragfbp} A {\it fragmentation} for a curve $\sigma:I \rightarrow {\mathbb R}^2$ corresponds to a finite sequence $0=t_0<t_1\ldots <t_m=s$ of elements in $I$ such that ${\mathcal L}(\sigma,t_{i-1},t_i)<  r$
with
$\sum_{i=1}^m {\mathcal L}(\sigma,t_{i-1},t_i) =s$
We denote by a {\it fragment}, the restriction of $\sigma$ to the interval determined by two consecutive elements in the fragmentation.
\end{definition}

The following results are presented for $1$-constrained curves and can be found in \cite{papera}. These give lower bounds for the length of curves when compared with arcs in unit circles and line segments. These results can be easily adapted for $\kappa$-constrained curves. Consider $\sigma (t)=(r(t)\cos \theta(t), r(t)  \sin \theta(t))$ in polar coordinates.

\begin{lemma}\label{rad} (cf. Lemma 2.8 in \cite{papera}) For any curve $\sigma :[0,s] \rightarrow {\mathbb R}^2$ with $\sigma(0)= (1,0)$, $r(t)\geq 1$, and $\theta(s)=\eta$, one has ${\mathcal L}(\sigma)\geq \eta$.
\end{lemma}

\begin{lemma}\label{seg} (cf. Lemma 2.9 in \cite{papera}) For any $C^1$ curve $\sigma :[0,s]\rightarrow {\mathbb R}^2$ with $\sigma(0)=(0,0)$, $\sigma(s)=(x,z)$ and $z\geq 0$, one has ${\mathcal L}(\sigma) \geq z$.
\end{lemma}

\begin{lemma}\label{ntp} (cf. Lemma 7.5 in \cite{paperc}) If a $1$-constrained curve $\sigma:[0,s] \rightarrow {\mathbb R}^2$ lies in a unit radius disk $D$, then either $\sigma$ is entirely in $\partial (D)$, or the interior of $\sigma$ is disjoint from $\partial ({D})$.
\end{lemma}

\section{A Fundamental Lemma}

We would like to emphasise that a $\kappa$-constrained plane curve has absolute curvature bounded above (when defined) by a positive constant $\kappa=\frac{1}{r}$ where $r>0$ is the minimum radius of curvature.

\begin{lemma}\label{r1} A $\kappa$-constrained plane curve $\sigma: I \rightarrow {\mathcal B}$ where,
$${\mathcal B}=\{(x,y)\in{\mathbb R}^2\,|\, -r<x<r\,\,,\,\,y\geq 0 \}$$
cannot satisfy both:
\begin{itemize}
\item $\sigma(0),\sigma(s)$ are points on the $x$-axis;
\item If $C$ is a radius $r$ circle with centre on the negative $y$-axis and $\sigma(0),\sigma(s)\in C$, then some point in ${Im}(\sigma)$ lies above $C$.
\end{itemize}
\end{lemma}

\begin{proof} Suppose such a curve exists. Let $h:{Im}(\sigma)\rightarrow {\mathbb R}_{\geq0}$ be the projection onto the $y$-axis. Since $\mbox{\it Im}(\sigma)$ is compact and $h$ is continuous, there exists $p\in I$ such that,
$$h(\sigma(p))=\max_{t\in I}{h(\sigma(t))}.$$
Consider a continuous one-parameter family of circles $C_u$ obtained by translating $C$ along the $y$-axis by $u\geq 0$. Note that by continuity there exists a $v\geq 0$ such that $\sigma$ lies inside $C_{v}$ and is tangent to $C_{v}$ at some point. By viewing $C_{v}$ as $\partial (D)$ in Lemma \ref{ntp} (nearby the point of tangency) we immediately obtain a contradiction.
\end{proof}

{ \begin{figure} [[htbp]
 \begin{center}
\includegraphics[width=1\textwidth,angle=0]{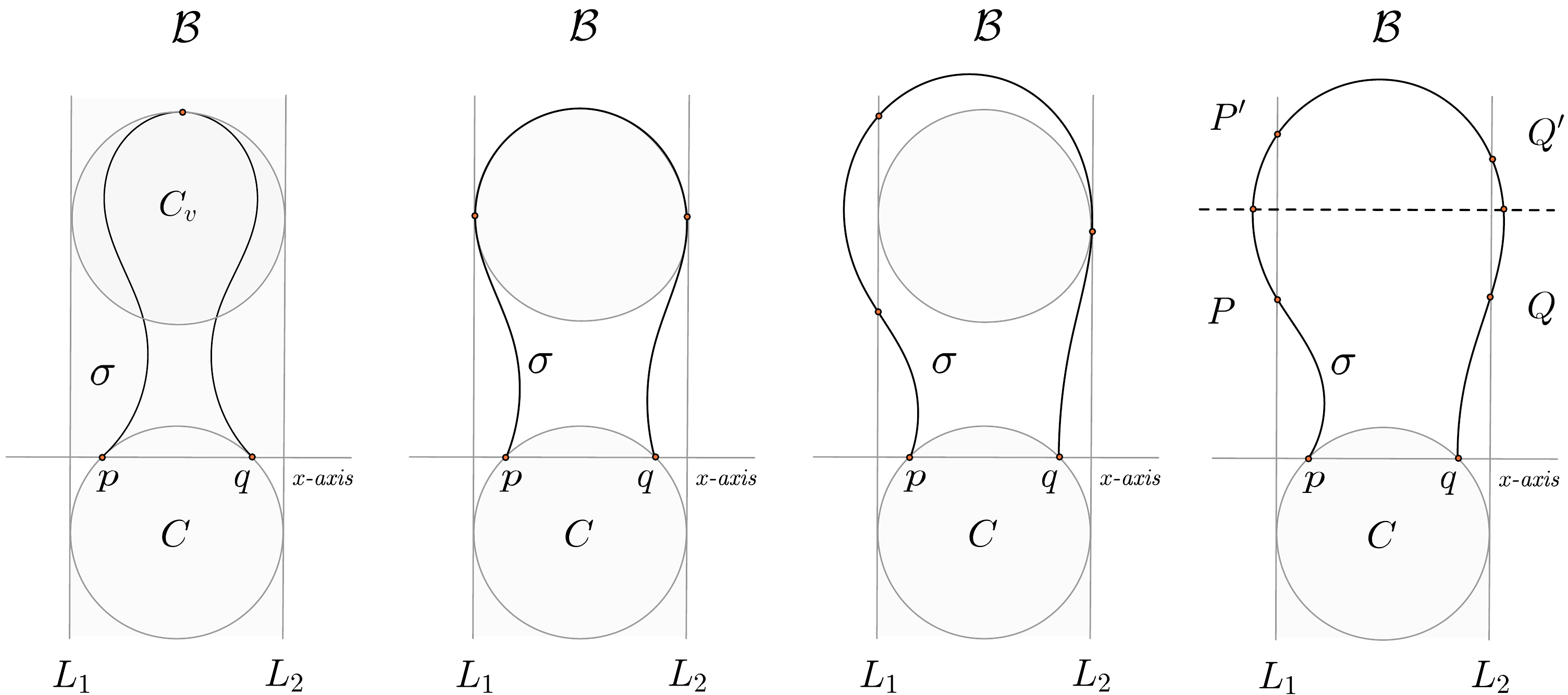}
\end{center}
\caption{An illustration of Corollary \ref{crossect}. Here $p$ and $q$ represent $\sigma(0)$ and ${\sigma(s)}$. We obtain the fourth illustration from the third one by clockwise rotating the band ${\mathcal B}$ with point of rotation the centre of $C$. In this fashion we obtain a pair parallel tangents. The dashed trace at the right corresponds to a cross section for $\sigma$.}
\label{figband}
\end{figure}}

\begin{definition} A plane curve $\sigma$ has {\it parallel tangents} if there exist $t_1,t_2 \in I$, with $t_1<t_2$, such that $\sigma'(t_1)$ and $\sigma'(t_2)$ are parallel and pointing in opposite directions.
\end{definition}

\begin{definition}\label{defncross} Let $L_1$ and $L_2$ be the lines $x=-r$ and $x=r$ respectively.  A line joining two points in $\sigma$ distant apart at least $2r$ one to the left of $L_1$ and the other to the right of $L_2$ is called a cross section (see Figure \ref{figband} right).
\end{definition}

Next result gives conditions for the existence of parallel tangents.

\begin{corollary} \label{crossect} Suppose a $\kappa$-constrained plane curve $\sigma: I \rightarrow {\mathbb R}^2$ satisfies:
\begin{itemize}
\item $\sigma(0),\sigma(s)$ are points on the $x$-axis.
\item If $C$ is a radius $r$ circle with centre on the negative $y$-axis, and $\sigma(0),\sigma(s)\in C$, then some point in ${Im}(\sigma)$ lies above $C$.
\end{itemize}
Then $\sigma$ admits parallel tangents and therefore a cross section.
\end{corollary}

\begin{proof} Consider a $\kappa$-constrained curve $\sigma$ satisfying the hypothesis given in the statement. By virtue of Lemma \ref{r1}, the curve $\sigma$ cannot be entirely contained in the band $\mathcal B$ (see Figure \ref{figband} left). It is not hard to see that if $\sigma$ is tangent to $L_1$ and $L_2$ from the inside of $\mathcal B$ then, a pair of parallel tangents is immediately obtained (see second illustration in Figure \ref{figband}). Suppose $\sigma$ is tangent to $L_2$ and crosses $L_1$ twice (see third illustration in Figure \ref{figband}). By rotating counterclockwise the parallel lines $L_1$ and $L_2$ simultaneously (a sufficiently small angle) these parallel lines cut $\sigma$ in at least in two points each line. Suppose $\sigma$ intersects $L_1$ at $P$ and $P'$, and that $\sigma$ intersects $L_2$ at $Q$ and $Q'$ (see Figure \ref{figband} right). Since $\sigma$ is $C^1$, then by applying the intermediate value theorem for the derivatives to $\sigma$ between $P$ and $P'$ and between $Q$ and $Q'$ we conclude the existence of parallel tangents. We ensure that the directions of the parallel vectors are of opposite sign by considering the sub arcs of $\sigma$ between the first time it leaves ${\mathcal B}$ and the first time it reenters $\mathcal B$ and then considering $\sigma$ between the last time it leaves ${\mathcal B}$ and the last time $\sigma$ reenters $\mathcal B$. Since $\sigma$ has a point to the left of $L_1$ and a point to the right of $L_2$, there exists a cross section, concluding the proof. 
\end{proof}

In general, it is not an easy task to construct $\kappa$-constrained homotopies between two given curves (see \cite{paperd}). In Proposition \ref{partan} we will see that the existence of parallel tangents leads to a method for constructing $\kappa$-constrained homotopies.

\begin{definition}Let $\sigma$ be a $C^1$ curve. The affine line generated by $\langle  \sigma'(t)\rangle$ is called the tangent line at $\sigma(t)$, $t\in I$. The ray containing $\sigma'(t)$ is called the positive ray.
\end{definition}

The next definition can be easily adapted for arc length-parametrised curves, we leave the details to the reader.

\begin{definition} \label{conc} Suppose that  $\sigma, \gamma: [0,1]\rightarrow {\mathbb R}^2$ with $\sigma(1)=\gamma(0)$. The concatenation of $\sigma$ and $\gamma $ is denoted by $\gamma\, \#\, \sigma$ and is defined to be,

\[ ({ \gamma\,\#\, \sigma})(t) = \left\{ \begin{array}{ll}
        {\sigma}(2t)\,\,\,\,\,\,\,\,\,\,\,\,\,\,\,\, 0\leq t\leq \frac{1}{2}\\
       {\gamma}(2t-1)\,\,\,\,\,\, \frac{1}{2}<t \leq 1.\end{array} \right. \]
\end{definition}

\begin{remark}\label{train} {\it(Train track displacement)}. The next result gives a direct method for obtaining $\kappa$-constrained homotopies. Suppose a $\kappa$-constrained curve $\sigma$ has parallel tangents at $t_1,t_2 \in I$. The tangent lines at $\sigma(t_1)$ and $\sigma(t_2)$ may work as {\it train tracks} for the displacement of the portion of $\sigma$ in between $\sigma(t_1)$ and $\sigma(t_2)$ (see Figure \ref{figpartan}). 
\end{remark}

\begin{proposition}\label{partan} The train track displacement obtained by the existence of parallel tangents in Remark \ref{train} induces a $\kappa$-constrained homotopy.
\end{proposition}

\begin{proof}Suppose $\sigma$ is a $\kappa$-constrained curve having parallel tangents at parameters $t_1$ and $t_2$. Consider the restriction $\tilde{\sigma}:[t_1,t_2]\rightarrow {\mathbb R}^2$. Subdivide $\sigma$ in such a way that $\sigma=\sigma_2\,\# \, \tilde{\sigma}\, \# \,\sigma_1$. Consider the parametrised lines $\psi_1,\psi_2: [0,1]\rightarrow {\mathbb R}^2$ defined by,
$$ \psi_1(r)=\tilde{\sigma}(t_1)(1-r)+Pr$$
 $$ \psi_2(r)=\tilde{\sigma}(t_2)(1-r)+Qr$$

 \noindent where $P$ belongs to the positive ray of $\langle  \sigma'(t_1)\rangle$ and $Q$ belongs to the negative ray of $\langle  \sigma'(t_2)\rangle$ with $d(P,\sigma(t_1))=d(Q,\sigma(t_2))$. Call $\tilde{\gamma}_r$ the translation of $\tilde{\sigma}$ obtained by adding the vector from $\tilde{\sigma}(t_1)$ to $ \psi_1(r)$. Define $\phi_r=\psi_3|_{[0,r]} \,\#  \, \tilde{\gamma}_r \,\# \, \psi_1|_{ \tiny{[0,r]}} $ for each $r\in (0,1]$, where $\psi_3(r)= \psi_2(1-r)$. Define the homotopy $\tilde{\mathcal H}(r)=\sigma_2\,\# \, \phi_r \, \# \,\sigma_1$. In this fashion we have that $\tilde{\mathcal H}(0)=\sigma_2 \, \# \, \phi_0 \, \# \, \sigma_1=\sigma$, and $\tilde{\mathcal H}(r)=\sigma_2 \, \# \, \phi_r \, \# \, \sigma_1=\gamma$ are both $\kappa$-constrained curves for $r\in [0,1]$ (after reparametrisation). 
\end{proof}

 The homotopy in Proposition \ref{partan} defines an operation on $\kappa$-constrained curves called operation of type {\sc III} (see Remark \ref{optype}). We say that the curves $\sigma$ and $\gamma$ in the previous result are {\it parallel homotopic}.

 \section{Homotopy classes in spaces of $\kappa$-constrained plane curves}

In order to determine the number of homotopy classes in $\Sigma(x,y)$ we first study the case where the initial and final points of a $\kappa$-constrained curve are different. Then we study closed $\kappa$-constrained curves. As a consequence of the curvature bound we have that if the initial and final points are different we have two scenarios, namely, $0<d<2r$ or $d\geq 2r$ where the minimum radius of curvature is $r=\frac{1}{\kappa}$. Soon we will see that if $0<d < 2r$ there exist two planar regions where no $\kappa$-constrained curve can be defined (see Figure \ref{figexampcurves}). In addition, if $0<d < 2r$ there exists a planar region that {\it traps} $\kappa$-constrained curves. That is, no $\kappa$-constrained curve defined in the trapping region can be made $\kappa$-constrained homotopic to a curve having a point in the complement of the trapping region. In particular, we conclude that these trapped curves correspond to a homotopy class of embedded $\kappa$-constrained curves.

{ \begin{figure}[htbp]
 \begin{center}
\includegraphics[width=.6\textwidth,angle=0]{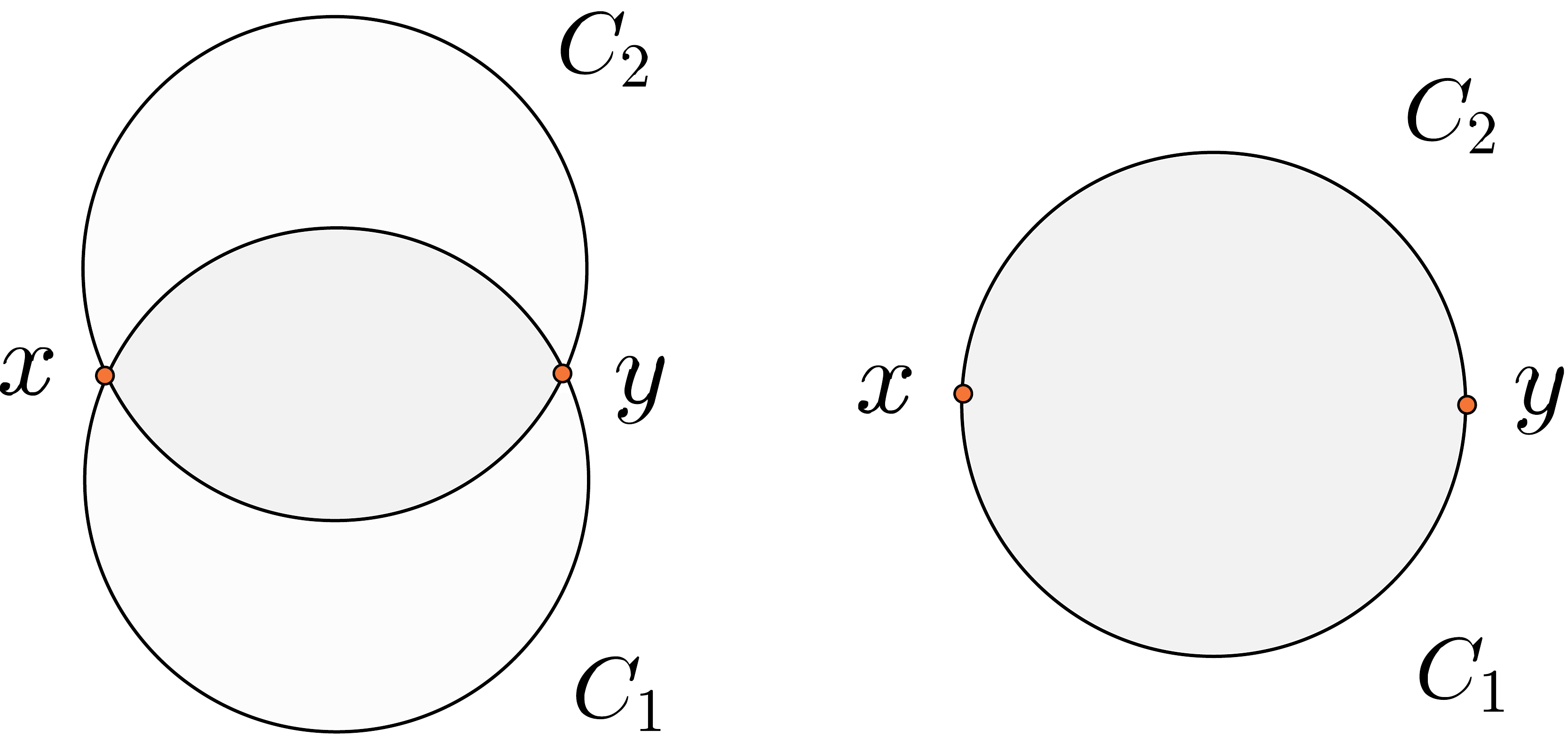}
\end{center}
\caption{Left: An illustration for $0 < d < 2r$. The region shaped like a lens correspond to $\mathcal I$. The union of the lighter shaded regions correspond to $\mathcal E$. Right: An illustration for $d = 2r$.}
\label{figless2r}
\end{figure}}

\begin{definition}\label{leq2r}  Suppose that $0<d<2$. Let $D_1$ and $D_2$ be radius $r$ disks ($C_1=:\partial (D_1)$ and $C_2=:\partial (D_2)$) and $C_1\cap C_2=\{x,y\}$. Set ${\mathcal I} ={int}( D_1\cap D_2)$ and  ${\mathcal U}=D_1\cup D_2$. Then define ${\mathcal E}={ int}({\mathcal U}) \setminus { cl}({\mathcal I})$. Also, denote by $\partial({\mathcal I})$ the union of the shorter circular arcs of $C_1$ and $C_2$ joining $x$ and $y$ (see Figure \ref{figless2r}).

\end{definition}

\begin{remark} When we study properties of $\mathcal I$, $\mathcal E$ or $\mathcal U$ we are implicitly saying that $\Sigma(x,y)$ is such that $x\neq y$ and $d<2r$.
\end{remark}

{ \begin{figure}[htbp]
 \begin{center}
\includegraphics[width=1\textwidth,angle=0]{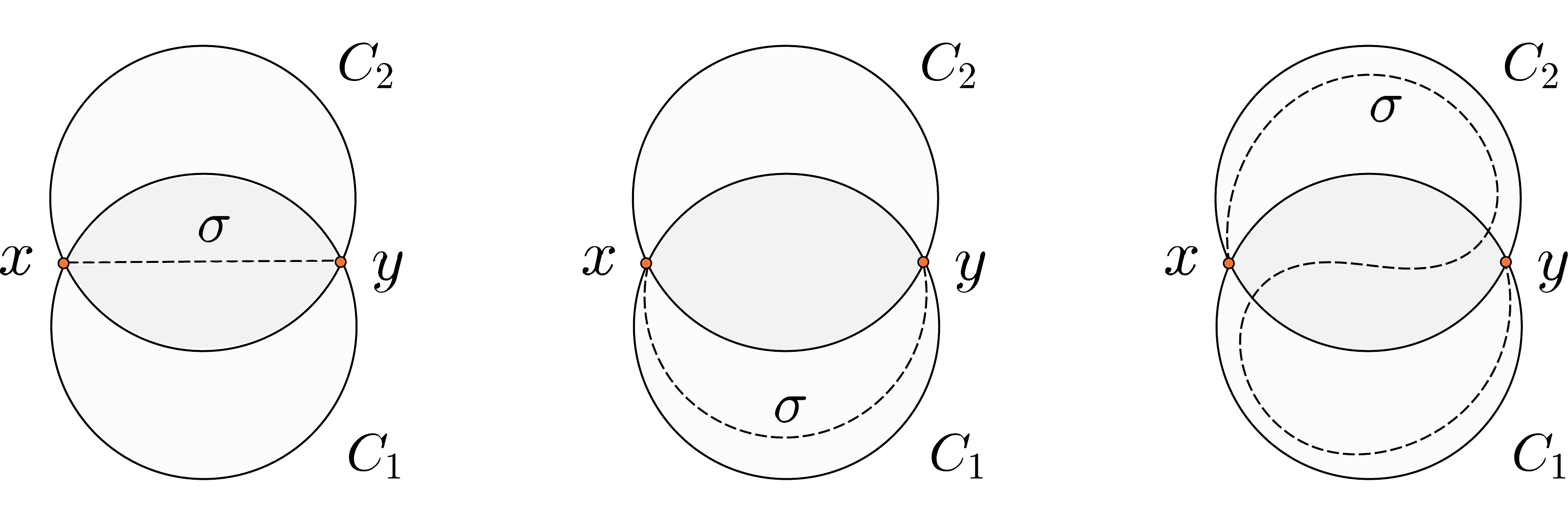}
\end{center}
\caption{Let $C_1$ and $C_2$ be radius $r$ circles. Does $\sigma$ represents a $\kappa$-constrained curve in ${\mathcal I}$, ${\mathcal E}$ or ${\mathcal U}$? (see Figure \ref{figless2r}).}
\label{figexampcurves}
\end{figure}}

\begin{remark} {\it (On piecewise constant curvature $\kappa$-constrained curves)} As seen in \cite{paperd}, constructing explicit $\kappa$-constrained homotopies is not a simple matter. In subsection \ref{redprocess} we will discuss a process applied to $\kappa$-constrained curves called {\it normalisation} (see \cite{papera, paperb, paperd}). The normalisation of a $\kappa$-constrained curve $\sigma$ is a piecewise constant curvature $\kappa$-constrained curve corresponding to a finite number of concatenated pieces called {\it components}. These components are arcs of radius $r$ circles and line segments. The number of components is called the {\it complexity} of the curve. It is important to note that both $\sigma$ and its normalisation are curves in the same connected component in $\Sigma(x,y)$. Our efforts in \cite{papera, paperb, paperd} have been made in order to overcome the difficulty of constructing explicit $\kappa$-constrained homotopies. After normalising $\sigma$ we apply a {\it reduction process}, also described in subsection \ref{redprocess}. This process consists on manipulating piecewise constant curvature $\kappa$-constrained curves while reducing length and complexity and without violating the curvature bound. 
\end{remark}

\begin{definition} A piecewise constant curvature $\kappa$-constrained curve whose circular components lie in radius $r$ circles is called $cs$ curve.
\end{definition}

With the intention of simplifying our arguments, in \cite{paperd}, we defined the so called {\it operations of type} {\sc I} and {\sc II}. The {\it Operations of type} {\sc III} will be defined to be the ones performed under the existence of parallel tangents (see Proposition \ref{partan}). Note that the first two operations are applied to $cs$ curves and the third one may be applied only to $\kappa$-constrained curves with parallel tangents.

\begin{remark}\label{optype} {\it (Operations on $cs$ curves, see \cite{paperd} and Figure \ref{figpartan}})
\begin{itemize} 
\item {\it Operations of type} {\sc I}: In order to perform operations of type {\sc I} we consider a point $z$ in the image of a $cs$ curve as rotation point. We then consider two radius $r$ disks (pushing disks) both tangent to the $cs$ curve at $z$. Once the rotation point is chosen, we twist the initial $cs$ curve along the boundary of the two pushing disks in a clockwise or counterclockwise fashion\footnote{For a full rotation see Figure 3 in \cite{paperd}.}.

\item An example of an {\it operation of type} {\sc II} is illustrated in Figure \ref{figpartan} upper-right; for an in-depth description, refer to \cite{paperd}.

\item {\it Operations of type} {\sc III}: These operations are defined to be the ones performed under the existence of parallel tangents (see Proposition \ref{partan}).
\end{itemize}
\end{remark}

{ \begin{figure}[htbp]
 \begin{center}
\includegraphics[width=1\textwidth,angle=0]{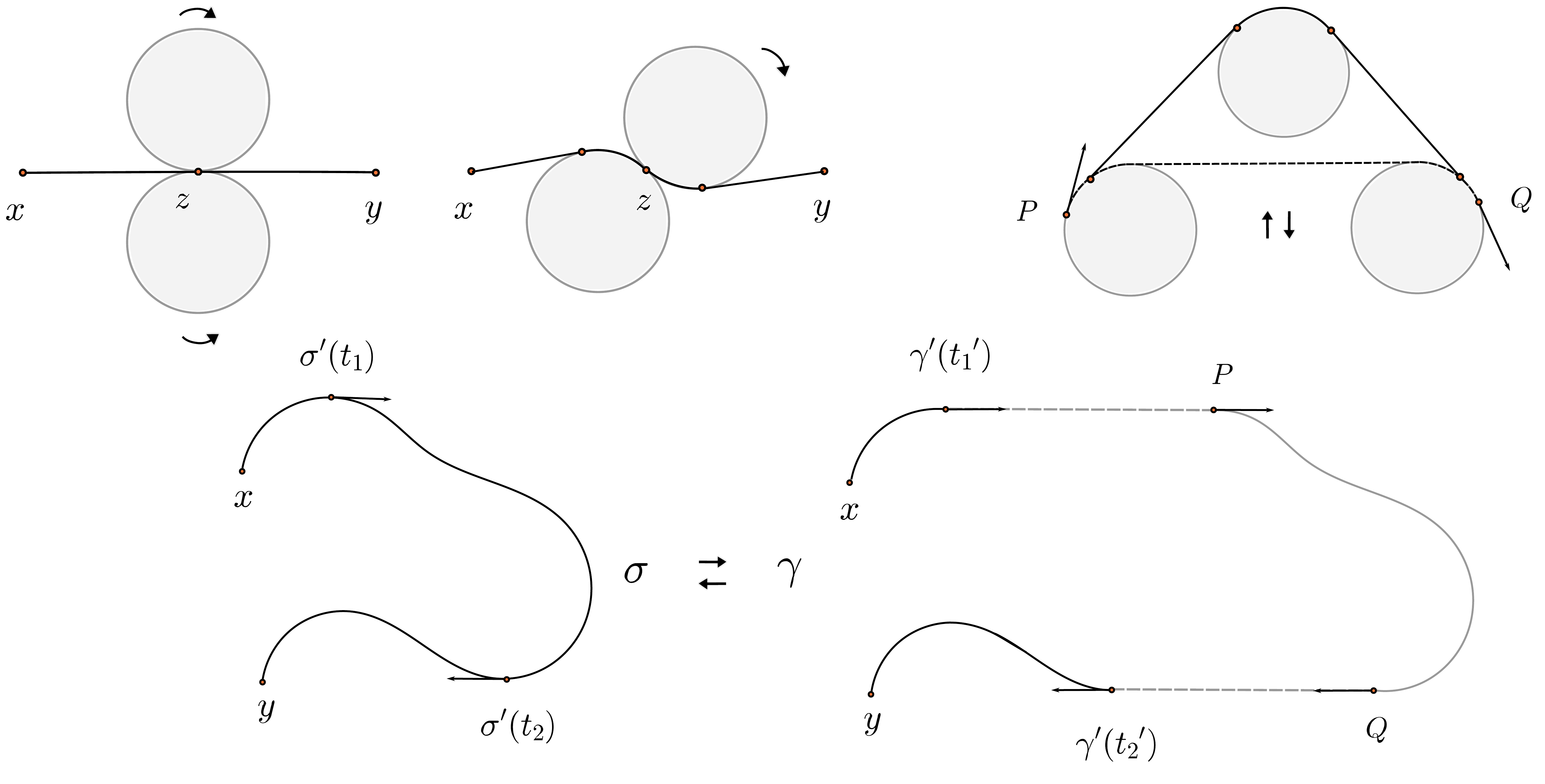}
\end{center}
\caption{Upper-left: An Illustration of an operation of type {\sc I} under a clockwise rotation. Upper-right: An Illustration of an operation of type {\sc II}. Lower: An Illustration of a train track displacement (an operation of type {\sc III}). Note that $\sigma$ has parallel tangents. In addition, observe that $\sigma'(t_1)=\gamma'(t_1')$ and $\sigma'(t_2)=\gamma'(t_2')$ for some parameters $t_1'$ and $t_2'$.}
\label{figpartan}
\end{figure}}



Next, we establish that the larger circular arcs connecting $x$ and $y$ in $C_1$ and $C_2$ (see Figure \ref{figless2r} left) can be deformed one to another without violating the prescribed curvature bound\footnote{An analogous construction was observed by N. Kuiper to L. E. Dubins in \cite{dubins 2} page 480 (proof not added in the manuscript).}.

\begin{proposition}\label{homotarcmov}Suppose that $0<d<2r$. The larger circular arcs in $C_1$ and $C_2$ joining $x$ and $y$ are $\kappa$-constrained homotopic in $\mathbb R^2$.
\end{proposition}

\begin{proof} We label each step in Figure \ref{fighomotarcmov} from left to right starting with 1 and finishing at 8. In step 1 we start with $C_2$. Since $0<d<2r$ we have that the length of $C_2$ is greater than $\pi r$, and in consequence, we have that $C_2$ has parallel tangents. In step 2 we applied Proposition \ref{partan} by performing an operation of type {\sc III}. In step 3 we consider the rotation point $z=x$ and start a clockwise operation of type {\sc I}. In steps 4 and 5 we rotate the pushing disk so that it coincides with $C_1$ in step 6. Since $C_1$ also has parallel tangents we apply an operation of type {\sc III} in step 6 to obtain the curve in step 7. We apply an operation of type {\sc II} to the curve in step 7 to obtain the curve in step 8. Since the curve in step 8 is symmetrical we can obtain the curve in step 8 by applying the same operations as before starting from $C_1$ in place of $C_2$ concluding the proof.
\end{proof}

{ \begin{figure}[htbp]
 \begin{center}
\includegraphics[width=1\textwidth,angle=0]{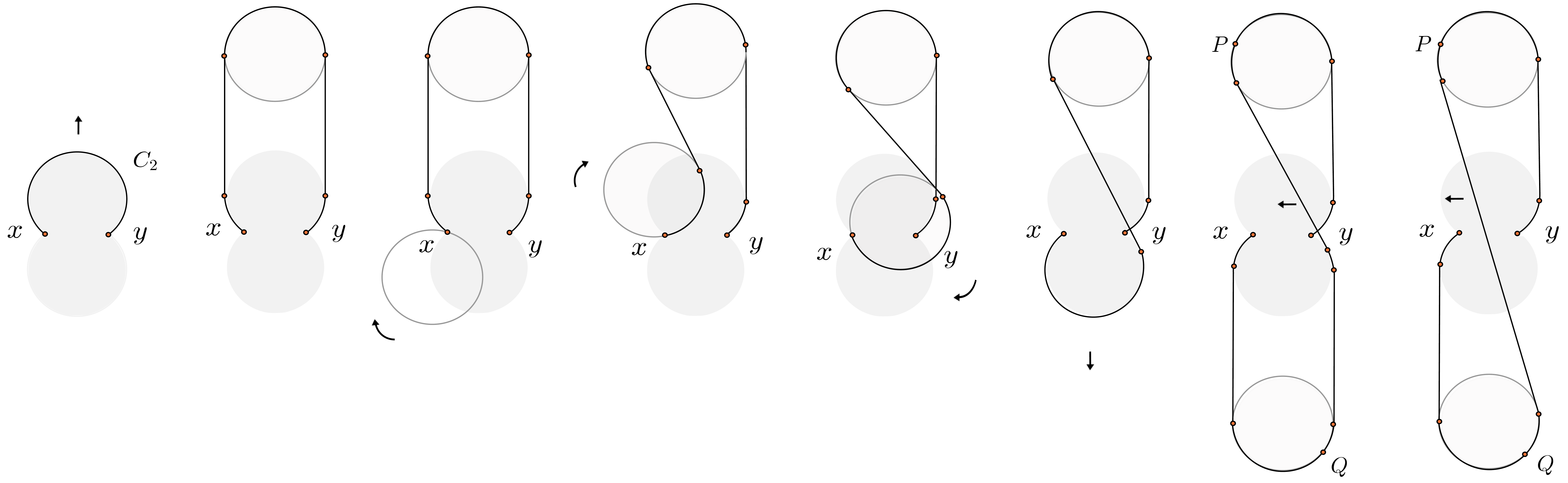}
\end{center}
\caption{The larger circular arcs joining $x$ and $y$ in $C_1$ and $C_2$ are $\kappa$-constrained homotopic.}
\label{fighomotarcmov}
\end{figure}}

\begin{definition} A curve $\sigma$ is said to be:
\begin{itemize}
\item {\it In} ${X}$ if $\sigma(t) \in {X}$ for all $t\in I$.
\item {\it Not in} ${X}$ if there exists $t\in I$ such that $\sigma(t) \notin {X}$.
\end{itemize}
\end{definition}

\begin{theorem}\label{below} A $\kappa$-constrained plane curve $\sigma:(0,s)\rightarrow \mathcal E$ with $\sigma \in \Sigma(x,y)$ cannot exist (see Figure \ref{figexampcurves}).
\end{theorem}

\begin{proof} Suppose there exists such a $\kappa$-constrained curve $\sigma$ (see Figure \ref{figexampcurves} centre). Consider an arc in $\partial({\mathcal I})$ as the arc of the circle $C$ from $p$ to $q$ in Lemma \ref{r1}. Corollary \ref{crossect} implies that $\sigma$ admits a cross section. The result follows as $\mathcal E$ is contained in $\mathcal B$, but the existence of a cross section implies that $\sigma$ cannot be contained in $\mathcal B$.
\end{proof}

\begin{corollary} The only $\kappa$-constrained plane curves {\it in} $\mathcal U$ are:
\begin{itemize}
\item the $\kappa$-constrained plane curves {\it in} $cl ({\mathcal I})$;
\item the $\kappa$-constrained plane curves having their image {\it in} $C_1$ or $C_2$.
\end{itemize}
\end{corollary}
\begin{proof} Suppose there exists a $\kappa$-constrained curve $\sigma$ {\it in} ${\mathcal U}$ having a point in the portion of $\mathcal E$ enclosed by $C_1$ and a point in the portion of $\mathcal E$ enclosed by $C_2$. By continuity, $\sigma$ has points lying in each of the arcs in $\partial({\mathcal I})$. By considering in turn each of the arcs in $\partial({\mathcal I})$ as $C$ in Lemma \ref{r1} the result follows.
\end{proof}

\begin{proposition}\label{trappedcbc} A $\kappa$-constrained plane curve {\it in} $cl({\mathcal I})$ not being an arc {\it in} $\partial({\mathcal I})$ cannot be made $\kappa$-constrained homotopic to a plane curve {\it in} $cl({\mathcal I})$ having a single common point with $\partial({\mathcal I})$ other than $x$ and $y$.

\end{proposition}
\begin{proof} Immediate from Lemma \ref{ntp}.
\end{proof}

\begin{proposition}\label{c} Consider $x,y\in {\mathbb R}^2$ so that $0<d<2r$. The arcs in $\partial({\mathcal I})$ cannot be made $\kappa$-constrained homotopic to a curve {\it not in} ${ cl}(\mathcal I)$. 
\end{proposition}
\begin{proof} Suppose there exists a $\kappa$-constrained homotopy between the upper arc in $\partial({\mathcal I})$ and a curve $\gamma$ not in ${ cl}(\mathcal I)$. Let $C$ be the radius $r$ circle containing the upper arc in $\partial({\mathcal I})$ (see Figure \ref{figband}). Since the curve $\gamma$ has a point not in ${cl}(\mathcal I)$ lying above $C$ (or below $C$, see Theorem \ref{below} and Figure \ref{figexampcurves} centre), then by applying Lemma \ref{r1} we obtain immediately a contradiction. 
\end{proof}

\begin{remark}\hfill
\begin{itemize}
\item If the distance between the endpoints satisfies $0<d<2r$, by virtue of Proposition \ref{trappedcbc} and Proposition \ref{c}, we have established a lower bound $>1$ for the number of homotopy classes in $\Sigma(x,y)$.
\item It is not hard to see that the arcs in $\partial({\mathcal I})$ can be made $\kappa$-constrained homotopic to the line segment joining $x$ and $y$. See \cite{paperd} for a rigorous continuity argument for the preservation of the curvature bound under continuous deformations.
\end{itemize}
\end{remark}

\begin{proposition} \label{segbound} Suppose $x\neq y$. The shortest $\kappa$-constrained plane curve in $\Sigma(x,y)$ is the line segment joining $x$ and $y$.
\end{proposition}
\begin{proof} Immediate from Lemma \ref{seg}.
\end{proof}

The following result can be found in \cite{paperb} for $1$-constrained curves. The proof for $\kappa$-constrained curves follows the same lines.

\begin{theorem}{\it (cf. Theorem 4.6 in \cite{paperb})} \label{loopbound} 
The shortest closed $\kappa$-constrained plane curve corresponds to the boundary of a radius $r$ disk.
\end{theorem}

In general, if $\sigma:I\to{\mathbb R}^2$ is a length minimiser, the image of $\sigma$ restricted to $[a,b]\subset I$ is also a length minimiser. From Theorem \ref{loopbound} we immediately have the following result.

\begin{corollary}\label{minoutside} If $0<d<2r$. The minimal length $\kappa$-constrained plane curves {\it not in} $cl({\mathcal I})$ are the longer arcs between $x$ and $y$ in $C_1$ and $C_2$.
\end{corollary}


\subsection{Normalising and reducing $\kappa$-constrained curves} \label{redprocess}

Here we discuss about crucial ideas presented in \cite{papera, paperb, paperd}. Recall from Definition \ref{fragfbp} that a fragmentation for a $\kappa$-constrained curve $\sigma$ corresponds to a partition of the image of $\sigma$ in a way that each piece, or {fragment}, has length less than $r=\frac{1}{\kappa}$. The idea is to consider fragments of length less than $r$ in order to allow the construction of a specially convenient type of curves called {\it replacements}. A replacement is a {\sc csc} curve i.e., a $\kappa$-constrained curve with fixed initial and final points and vectors corresponding to a concatenation of three consecutive pieces, being these, an arc in a radius $r$ circle followed by a line segment followed by an arc in a radius $r$ circle. The following two results are of importance. 

 \begin{proposition}\label{homotopyfragment} (Proposition 3.6. in \cite{paperd}) A fragment is bounded-homotopic to its replacement.
\end{proposition}

 \begin{lemma}\label{fundlemma} (Lemma 2.14 in \cite{papera}) The length of a replacement is at most the length of the associated fragment with equality if and only if these are identical.
\end{lemma}

 The {\it normalisation process} replaces any $\kappa$-constrained curve $\sigma$ with a prescribed fragmentation by a $cs$ curve, called its {\it normalisation}, which is $\kappa$-constrained homotopic to $\sigma$. We $\kappa$-constrained homotope each fragment to a {\sc csc} curve (the replacement). Note that the complexity of the normalisation will depend on the fragmentation. The {\it reduction process} corresponds to a sequence of $\kappa$-constrained homotopies so that at each step an initial $cs$ curve is $\kappa$-constrained homotoped to a non-longer $cs$ curve having no higher complexity than the initial one. We start with the normalisation, and after a finite number of steps, we end up with a length minimiser in the homotopy class of $\sigma$ (see \cite{paperd}).

\begin{proposition} \label{csformfbp} A $\kappa$-constrained plane curve $\sigma$ is $\kappa$-constrained homotopic to a $cs$ curve of length at most the length of $\sigma$.
\end{proposition}
\begin{proof} Consider a fragmentation for $\sigma \in \Sigma(x,y)$. Consider for each fragment a replacement which by Lemma \ref{fundlemma} is of length at most the length of the fragment. Then we apply Proposition \ref{homotopyfragment} to conclude that each fragment is bounded-homotopic to its replacement. After a reparametrisation we obtain that $\sigma$ is $\kappa$-constrained homotopic to a $cs$ curve of length at most the length of $\sigma$.
\end{proof}

\begin{theorem}\label{closedhomot} The space $\Sigma(x,x)$ corresponds to a single homotopy class of $\kappa$-constrained plane curves (see Figure \ref{figclosedminmov}).
\end{theorem}
\begin{proof} Consider a fragmentation for $\sigma \in \Sigma(x,x)$. By applying the reduction process in \cite{paperd} to $\sigma$ we obtain that $\sigma$ is $\kappa$-constrained homotopic to a minimal length element in its homotopy class i.e, a radius $r$ circle containing the base point $x$ (cf. Theorem \ref{loopbound}). On the other hand, by considering a different element $\gamma \in \Sigma(x,x)$ and by applying the reduction process to $\gamma$, we conclude that $\gamma$ is also $\kappa$-constrained homotopic to a minimal length element in its homotopy class\footnote{It is easy to see that there is an infinite number of such circles all of them $\kappa$-constrained homotopic one to another.}. Therefore, by transitivity, we conclude that $\sigma$ and $\gamma$ are $\kappa$-constrained homotopic.
\end{proof}

{ \begin{figure}[htbp]
 \begin{center}
\includegraphics[width=1\textwidth,angle=0]{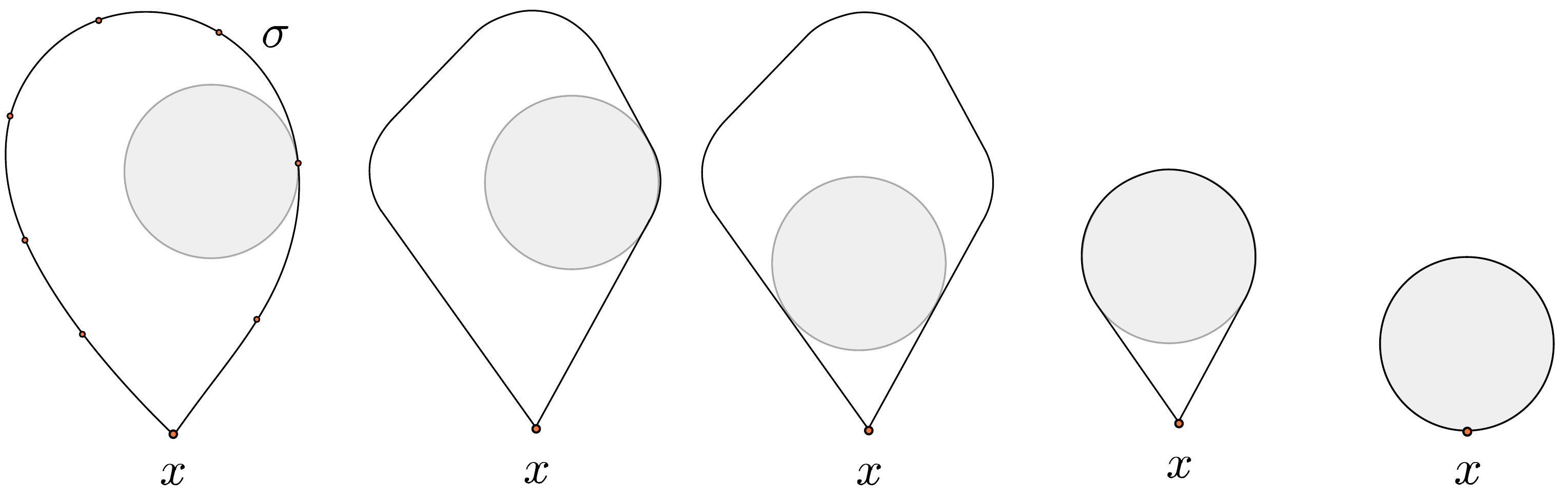}
\end{center}
\caption{An illustration of Theorem \ref{closedhomot} via the reduction process in Remark \ref{redprocess}. A curve $\sigma \in \Sigma(x,x)$. The dots are images of points in a fragmentation. Each fragment is $\kappa$-constrained homotopic to a {\sc csc} curve. A closed $cs$ curve is $\kappa$-constrained homotopic to a radius $r$ circle.}
\label{figclosedminmov}
\end{figure}}

\begin{remark} Recall from Definition \ref{leq2r} that for $x,y\in {\mathbb R}^2$ with $0<d<2$ we have the set ${\mathcal I} ={int}( D_1\cap D_2)$. Here $D_1$ and $D_2$ are the two radius $r$ circles containing both $x$ and $y$ in their boundaries.
\end{remark}

\begin{theorem} \label{Iconnect}Choose $x,y\in {\mathbb R}^2$ so that $0<d<2r$. Then the space of $\kappa$-constrained plane curves {\it in} $cl({\mathcal I})$ correspond to a homotopy class of embedded curves in $\Sigma(x,y)$.
\end{theorem}
\begin{proof} Consider a fragmentation for $\sigma \in \Sigma(x,y)$ {\it in} $cl({\mathcal I})$. By applying to $\sigma$ the reduction process described in Remark \ref{redprocess} we obtain that $\sigma$ is $\kappa$-constrained homotopic to the unique minimal length element in its homotopy class i.e., the line segment joining $x$ and $y$ (cf. Theorem \ref{segbound}). On the other hand, by considering a fragmentation for different element $\gamma \in \Sigma(x,y)$ {\it in} $cl({\mathcal I})$ and by applying the reduction process to it, we conclude that $\gamma$ is also $\kappa$-constrained homotopic to the minimal length element in its homotopy class. Therefore, by transitivity, we conclude that $\gamma$ and $\sigma$ are $\kappa$-constrained homotopic curves. To check that such curves are actually embedded, let $\sigma \in \Sigma(x,y)$ be {\it in} $cl({\mathcal I})$ having self intersections. Consider $\sigma$ in between the first self intersection. By the Pestov-Ionin Lemma (\cite{pestov}) there exists a radius $r$ disk in the interior component of $\sigma$ in between the considered self intersection. By virtue of Corollary \ref{crossect} we have that $\sigma$ has a cross section. Since $d>0$ we have that the diameter of $cl(\mathcal I)$ is lesser than $2r$ implying that $\sigma$ is a curve {\it not in} $cl({\mathcal I})$ leading to a contradiction.
\end{proof}

Note that if $0<d<2r$ we have determined that $|\Sigma(x,y)|\geq 2$. Next results proves that indeed $|\Sigma(x,y)|= 2$.

\begin{theorem}\label{Outconnect}Choose $x,y\in {\mathbb R}^2$ so that $0<d<2r$. Then the space of $\kappa$-constrained plane curves {\it not in} $cl({\mathcal I})$ is a homotopy class in $\Sigma(x,y)$.
\end{theorem}

\begin{proof} Consider a fragmentation for a curve $\sigma \in \Sigma(x,y)$ {\it not in} $cl({\mathcal I})$. By applying the reduction process in Remark \ref{redprocess} to $\sigma$ we obtain that $\sigma$ is $\kappa$-constrained homotopic to the minimal element in its homotopy class. By virtue of Proposition \ref{c} such a curve cannot be the line segment joining $x$ and $y$ since the latter is {\it in} $cl(\mathcal I)$. By Corollary \ref{minoutside} the minimal $\kappa$-constrained curve {\it not in} $\mbox{\rm cl}({\mathcal I})$ must be the longer arc joining $x$ and $y$ in $C_1$ or $C_2$.  By considering different element $\gamma \in \Sigma(x,y)$ {\it not in} $cl({\mathcal I})$ and by applying the reduction process to it, we conclude that $\gamma$ is also $\kappa$-constrained homotopic to a minimal length element in its homotopy class i.e., one of the longer arcs in $C_1$ or $C_2$ joining $x$ and $y$. By Proposition \ref{homotarcmov} the longer arcs joining $x$ and $y$ in $C_1$ and $C_2$ are $\kappa$-constrained homotopic. By transitivity, we conclude that $\gamma$ and $\sigma$ are $\kappa$-constrained homotopic.
\end{proof}

\begin{theorem}\label{Oneconnect} If $d \geq 2r$. Then $|\Sigma(x,y)|=1$.
\end{theorem}
\begin{proof} The proof is identical as in Theorem \ref{Iconnect}.
\end{proof}

\section{Main result}

\begin{theorem}\label{mainthmcbc} Choose $x,y \in {\mathbb R}^2$. Then:
\[  |\Sigma(x,y)| = \left\{ \begin{array}{lll}
        1\,\,\,\,\,\,\,\,\,\,\,\,\,\,\,\,\,\,\,\,\,\,\,\,\,\,\,\,  d=0\\
        2\,\,\,\,\,\,\,\,\,\,\,\,\,\,\,\,\, 0<d<2r\\
       1\,\,\,\,\,\,\,\,\,\,\,\,\,\,\,\,\,\,\,\,\,\,\,\,\,\,\,\,  d\geq 2r \end{array} \right. \]
\end{theorem}

\begin{proof} If $d=0$, by virtue of Theorem \ref{closedhomot}, we conclude that $|\Sigma(x,y)|=1$. If $0<d<2$, by applying Theorem \ref{Iconnect} and Theorem \ref{Outconnect}, we conclude that $|\Sigma(x,y)|=2$. If $d\geq 2r$, by virtue of Theorem \ref{Oneconnect} we conclude $|\Sigma(x,y)|=1$ concluding the proof.
\end{proof}

In other words, whenever $d=0$ any two closed $\kappa$-constrained curves are $\kappa$-constrained homotopic one to another. On the other hand, if $0<d<2r$ any two $\kappa$-constrained curves are $\kappa$-constrained homotopic one to another if and only if they are either both {\it in} $ cl(\mathcal I)$ or both {\it not in} $cl(\mathcal I)$. Finally, if $d \geq 2r$ any two $\kappa$-constrained curves are $\kappa$-constrained homotopic one to another.

\section{Acknowledgments}

I would like to thank both reviewers for their thorough comments and
suggestions, particularly the second reviewer for the many efforts on behalf of the manuscript.

\bibliographystyle{amsplain}
  
\end{document}